\documentclass[11pt]{article}

\usepackage[margin=1.5in]{geometry}

\usepackage{amssymb,amsmath,amsthm}
\usepackage{fancyhdr}
\usepackage[dvipsnames]{xcolor}
\usepackage{indentfirst}
\usepackage{url}
\usepackage{verbatim}
\usepackage{asymptote}
\usepackage{hyperref}
\usepackage{graphicx}

\usepackage[justification=centering]{caption}

\theoremstyle{plain}
\newtheorem{theorem}{Theorem}[section]
\newtheorem{corollary}[theorem]{Corollary}
\newtheorem{lemma}[theorem]{Lemma}
\newtheorem{proposition}[theorem]{Proposition}

\newtheorem*{principle}{Main Principle}

\theoremstyle{definition}
\newtheorem{definition}[theorem]{Definition}
\newtheorem{example}[theorem]{Example}

\theoremstyle{remark}

\newtheorem*{remark*}{Remark}

\renewcommand{\L}{\text{L}}

\title{The Penney's Game with Group Action}
\author{Tanya Khovanova, Sean Li}
\date{}
\begin{document}

\maketitle

\begin{abstract}
Consider equipping an alphabet $\mathcal{A}$ with a group action that partitions the set of words into equivalence classes which we call patterns. We answer standard questions for the Penney's game on patterns and show non-transitivity for the game on patterns as the length of the pattern tends to infinity. We also analyze bounds on the pattern-based Conway leading number and expected wait time, and further explore the game under the cyclic and symmetric group actions.
\end{abstract}

\section{Introduction}

The \textbf{Penney ante game}, or \textbf{Penney's game}, is a two-player game with a fair coin. The two players, Alice and Bob, each pick a word consisting of $H$s and $T$s, where $H$ is for heads and $T$ is for tails. Both words are of the same length. The coin is then flipped repeatedly, and the person who chose the word that appears first is declared the winner. Fixing the words that Alice and Bob choose, what are the odds that Alice wins? 

Many other interesting questions are related to this game. What is the expected amount of flips, also known as the expected wait time, until Alice's word appears? How many words of a given length avoid Alice's word? Bob's word? Both? What is the best choice for Bob if he knows Alice's word?

The game first appeared in 1969 as a problem submitted to the \textit{Journal of Recreational Mathematics} by Walter Penney \cite{WP}. It was later popularized by Martin Gardner \cite{MGSA,MG}, who introduced the Conway leading numbers that allow us to easily calculate the odds for the game. Not only that, but many things can be expressed through Conway leading numbers. For example, the expected wait time for a particular word is twice the Conway leading number. The same method is described in the \textit{Winning Ways for Your Mathematical Plays} \cite{BCG}. Collings in 1982 \cite{C} generalized Conway leading numbers to the case when we do not start from scratch, but from a given string.

This game can be generalized to alphabets consisting of more than two letters. The generalization of these results to larger alphabets was explored by Guibas \& Odlyzko in 1981 \cite{GO}.

Penney's game is a famous classic example of a \emph{non-transitive game}. For example, the word $HHT$ has better odds than $HTT$, the word $THH$ has better odds than $HHT$, while the word $THH$ does not have better odds than $HTT$. In addition, Bob can always choose a word with better odds that Alice, no matter what word she chooses. The best choice for Bob for any alphabet was discussed in \cite{GO} and later finalized by Felix in 2006 \cite{F}.

We cover all the details related to the original game in Section~\ref{sec:originalproblem}.

Although not within the scope of this paper, there are many preexisting variations of Penney's game. A popular variation uses various other objects in place of a coin, such as a finite deck of cards in Humble \& Nishiyama \cite{HN} or a roulette wheel in Vallin \cite{V}. Many more variations of this game were studied by Agarwal et al.~\cite{STEPPenney}.

In this paper, we introduce a variation of this game where Alice chooses a pattern, representing a collection of words, rather than a single word. For example, she can bet that three identical flips will appear first. In response, Bob can bet that three flips that form an alternating sequence of heads and tails will appear first instead.

Formally, we define a group action on the alphabet that swaps $H$ and $T$ and lets two words be equivalent if they are within the same orbit of the group action. In our particular example above, Alice bets on the two words $HHH$ and $TTT$, while Bob bets on the two words $HTH$ and $THT$.

We generalize this example to any group action. The goal of the paper is to generalize the classic Penney's game within this new framework. To do this, we calculate the odds of winning depending on the chosen patterns, and also expand upon many known results for the original game. We present the following road-map of the paper.

We start Section~\ref{sec:piwga} with a motivating example, then define the group action and the notion of a pattern. We also list interesting potential groups for study. 

In Section~\ref{sec:GroupActionAndCorrelation}, we define the correlation, correlation polynomial, period, and Conway leading number for two patterns. We make a few expository claims about their properties and discuss the similarities and differences between the theory of patterns and the theory of words.

In Section~\ref{sec:gfs}, we prove the theorem that describes the generating functions that avoid a given set of patterns and also the functions when a particular pattern appears for the first time while avoiding other patterns.

In Section~\ref{sec:ewt}, we calculate the expected wait time of a given pattern. We discuss patterns of fixed length with maximal and minimal wait times.

In Section~\ref{sec:odds}, we calculate the odds of Alice winning the game, given the choice of patterns for Alice and Bob, thus generalizing Conway's formula. We use this result to also compute the expected length of the game on two patterns.

In Section~\ref{sec:cyclic}, we concentrate on the cyclic group and the group action that cycles all the letters of the alphabet. We show via bijection that the game on patterns under a cyclic group is equivalent to the game on words that are one letter shorter.

In Section~\ref{sec:symmetric}, we concentrate on a symmetric group: the group that shuffles all the letters of the alphabet. We provide many examples and calculate the exact lower bound for Conway leading numbers. We find that for this group, it is not always possible for Bob to find a word with better odds than Alice's word. More precisely, there exist words that have better odds than other words of the same length.

\section{The original problem}\label{sec:originalproblem}

In this section, we cover what is known about Penney's game. For more detail one can check \cite{BCG,C,F,MGSA,MG,GO,WP}.

The \textbf{Penney's game} is a two-player game with a fair coin. The two players, Alice and Bob, each pick a word consisting of $H$s and $T$s. The coin is then flipped repeatedly, and the person who chose the word that appears first is declared the winner. For instance, say that Alice and Bob pick $HHH$ and $HTH$, respectively. Then in the sequence of flips $HHTTHTTT\underline{HTH}$, Bob wins on the eleventh flip. A natural question then arises. Given Alice's and Bob's words, what are the odds that either player wins?

Penney's game is a classic \emph{non-transitive game}. Here is an example of a loop of preferences:
\begin{itemize}
    \item if Alice picks $HTT$, Bob picks $HHT$ to have a higher chance of winning;
    \item if Alice picks $HHT$, Bob picks $THH$ to have a higher chance of winning;
    \item if Alice picks $THH$, Bob picks $TTH$ to have a higher chance of winning; and
    \item if Alice picks $TTH$, Bob picks $HTT$ to have a higher chance of winning.
\end{itemize}

Moreover, if Bob can choose his word after he knows Alice's word, he always has a winning strategy regardless of the word Alice chooses. For words of length 3, Bob can always find a word that makes his odds at least 2 to 1, see for example~\cite{BCG,F}. The fact that Bob can always choose a word of the same length as Alice's word with better odds implies non-transitivity.

\subsection{Correlation polynomials and Conway leading numbers}

The explicit probabilities of winning are also known for any pair of words.  In the literature, this theory is expressed in terms of correlation polynomials, or equivalently, 
Conway leading numbers. 

We first fix an alphabet $\mathcal{A}$ of $q$ letters. In the classical Penney's game, the alphabet consists of two letters: $H$ and $T$.

Consider a word $w = w(1)w(2)\dots w(\ell)$ of length $\ell$ in our alphabet $\mathcal{A}$.
The substring of letters from $w$ between the $i$-th and $j$-th letter inclusive,  that is $w(i)w(i+1) \dots w(j)$ is denoted by $w(i,j)$. In this paper we are mostly interested in prefixes $w(1,j)$ and suffixes $w(\ell-j+1,\ell)$ of length $j$.

We now define the \textbf{autocorrelation} of a word $w$ of length $\ell$, that is the correlation of a word with itself.

The autocorrelation vector $C(w,w)$ of $w$ is a vector $(C_{0},\dots ,C_{\ell-1})$, where $C_{i}$ is equal to $1$ if the first $\ell-i$ letters match the last $\ell-i$ letters, i.e.
\[w(1,\ell-i) = w(i+1,\ell),\]
and is equal to $0$ otherwise. Note that $C_0 = 1$ for any word. The \textbf{autocorrelation polynomial} of $w$ is defined as $c_{w,w}(z)=C_{0}z^{0}+\dots +C_{\ell-1}z^{\ell-1}$. It is a polynomial of degree at most $\ell-1$. As $C_0 = 1$, we have $C_{w,w}(0) = 1$ for all words $w$.

\begin{example}
The autocorrelation vector $C(HTHT,HTHT)$ is $(1,0,1,0)$. The corresponding autocorrelation polynomial is $1+x^2$.
\end{example}

The \textbf{Conway leading number (CLN)} for a given word $w$ is defined as the value of the correlation vector viewed as a string written in base $q$. The Conway leading number is denoted as  $w\L w$. Namely, the Conway leading number equals
\[w\L w= q^{\ell-1}C_{w,w}\left(\frac{1}{q}\right).\]

\begin{example}
With an alphabet of size $2$, the Conway leading number of $w = HTHT$ is $w\L w = 1010_2 = 10$.
\end{example}

Note that the autocorrelation vector and polynomial do not depend on the size of the alphabet, while the CLN does.

If $C_i = 1 > 0$ for a word $w$, then
\[w(1)w(2) \dots w(\ell-i) = w(i+1)w(i+2)\dots w(\ell).\]
Equivalently, the word $w$ has period $i$: $w(k) = w(k+i)$ for all $1 \leq k \leq \ell - i$.

The coordinates of the autocorrelation vector depend on each other. 
If a word has period $s$, it also has period $ks$ for all $k \geq 1$. If a word has periods $s$ and $t$, it also has period $s+t$. For example, if $C_1 = 1$, then all the letters in the word are the same, which implies that $C_i = 1$ for all $i < \ell$.

Generally, if a word has periods $s < t$ it does not necessarily have period $t-s$.
\begin{example}
The word $HTHHTH$ has periods $3$ and $5$, but not period $5-3 = 2$.
\end{example}
However, for sufficiently long words the implication is true. The following proposition and its corollary are proven in~\cite{GO}.

\begin{proposition}
If a word $w$ of length $\ell \geq s+t$ has periods $s<t$, then $w$ also has period $t-s$.
\end{proposition}

\begin{corollary}
\label{cor:period}
If a word $w$ of length $\ell$ has least period $s$ and period $t$ not divisible by $s$, then $t \geq \lfloor (\ell+1)/2 \rfloor + 1$.
\end{corollary}

Now we define the correlation polynomial between two words $v$ and $w$ that do not have to be the same length. Let $\ell$ be the length of $v$. The correlation vector $C(v,w)$ is defined as $C=(C_{0},\dots ,C_{\ell-1})$, with $C_{i}$ being 1 if the suffix of $v$ of length $\ell-i$ equals the prefix of $w$ length $\ell-i$, i.e. $v(i-1,\ell) = w(1,\ell-i)$, and 0 otherwise. Then the correlation polynomial of $v$ and $w$ is defined as $C_{v,w}(z)=C_{0}z^{0}+\dots +C_{\ell-1}z^{\ell-1}$. It is a polynomial of degree at most $\ell-1$. Similar to before, the Conway leading number between two words is defined as the value of the correlation vector interpreted as a string in base $q$ and it is denoted as $v\L w$. Specifically, in terms of the correlation polynomial, the Conway leading number is 
\[v\L w = q^{\ell-1} C_{v,w}\left(\frac{1}{q}\right).\]

\begin{example}
The correlation between the words $HTH$ and $HTHT$ is the vector $C(HTH, HTHT) = (1, 0, 1)$, and the correlation polynomial is $C_{HTH,HTHT}(z) = 1 + z^2$. The Conway leading number between the two words is $101_2 = 2^2 (1+\frac1{2^2}) = 5$. Note that $C(HTHT, HTH) = (1,0,1,0)$, so in general correlation is \emph{not} commutative. 
\end{example}

\subsection{Generating functions}

A set of words is \emph{reduced} if no word $w$ is a substring of another word $w'$ in the set. For instance, the set $\{HTH,TTHTH\}$ is not reduced.

Suppose we have a reduced set of $k$ words $S = \{w_1,w_2,\ldots,w_k\}$ with lengths $\ell_1, \ell_2, \dots, \ell_k$, composed of letters from an alphabet of size $q$. Let $A(n,S)$ denote the number of strings of length $n$ which avoid all words in $S$. We define 
$$G(z,S) = \sum_{k=0}^\infty A(k,S) z^{k}$$
to be the generating function which describes the number of words avoiding all words in $S$. Similarly, let $T_{w_i}(n,S)$ denote the number of strings of length $n$ which avoid all words in $S$, except for a final appearance of $w_i$ at the end of the word. We call such strings \textbf{first occurrence} strings. Then we define the generating function
$$G_{w_i}(z,S) = \sum_{k=0}^\infty T_{w_i}(k,S) z^{k}.$$
When the set $S$ in question is obvious, we drop it to shorthand, so $G(z) = G(z,S)$ and $G_{w_i}(z) = G_{w_i}(z,S)$.

\begin{remark*}
The reason we reduce the set $S$ is the following. Suppose that $w$ is a substring of $w'$. Then whenever $w'$ appears at the end of a string, then $w$ will appear too, i.e. $G_{w'}(z, S) = 0$ or a degeneracy.
\end{remark*}

The following theorem on the avoiding set $S$ is proven in~\cite{GO}.

\begin{theorem}
\label{thm:system}
The generating functions $G(z)$, $G_{w_1}(z)$, $G_{w_2}(z)$, $\dots$, $G_{w_k}(z)$ satisfy the following system of linear equations:
$$(1-qz)G(z) + G_{w_1}(z) + G_{w_2}(z) + \dots + G_{w_k}(z) = 1$$
$$G(z) - z^{-\ell_1}C_{w_1, w_1}(z) G_{w_1}(z) - \dots - z^{-\ell_k}C_{w_k, w_1}(z)G_{w_k}(z) = 0$$
$$G(z) - z^{-\ell_1}C_{w_1, w_2}(z) G_{w_1}(z) - \dots - z^{-\ell_k}C_{w_k, w_2}(z)G_{w_k}(z) = 0$$
$$\vdots$$
$$G(z) - z^{-\ell_1}C_{w_1, w_k}(z) G_{w_1}(z) - \dots - z^{-\ell_k}C_{w_k, w_k}(z)G_{w_k}(z) = 0$$
\end{theorem}

The following corollary is also proven in~\cite{GO}.

\begin{corollary}
If $k=1$, i.e. our set $S$ consists of a single word $w$, we have $$G(z) = \frac{C_{w,w}(z)}{z^{\ell} + (1-qz) C_{w,w}(z)},$$
$$G_{w}(z) = \frac{z^{\ell}}{z^{\ell} + (1-qz) C_{w,w}(z)}.$$
\end{corollary}

\begin{remark*}
Observe that the denominator is the same for both $G(z)$ and $G_w(z)$, implying that the corresponding sequences follow the same recurrence relations with different initial terms.
\end{remark*}

\begin{example}
For a word $w=HH$ with $q=2$ and $C(w,w) = 1 + z$, we have 
$$G(z) = \frac{1 + z}{1 - z - z^2}, \quad G_{w}(z) = \frac{z^2}{1 - z - z^2}.$$
The coefficients follow the same recurrence as the Fibonacci numbers.
\end{example}

\subsection{Expected wait time}

When our set $S$ consists of one word $w$, then $G_w(z)$ is the generating function describing the number of string that end with $w$ and do not contain $w$ otherwise. Thus, the expected wait time is $zG'_w(z)$ evaluated at $z=\frac{1}{q}$. The result is the following formula for the expected wait time:
\[q^{\ell}C_{w,w}\left(\frac{1}{q}\right) = q\cdot w\L w.\]
This gives us a closed form for the expected wait time for any word in terms of its autocorrelation. Note that in this case, we are mostly interested where $q=2$, the expected wait time is twice the Conway leading number.

\begin{example}
For the word $w = HTHT$ that we discussed before, we have $w\L w = 1010_2 = 10$. The expected wait time for the word $HTHT$ is twice the Conway leading number which can also be calculated using the autocorrelation polynomial as $2^4(1 + \frac{1}{2^2}) = 20$. 
\end{example}

\subsubsection{Bounds on the wait time}

As $C_0=1$ for any word, the shortest expected wait time for a word of length $\ell$ is $q^\ell$. We call such words \textbf{non-self-overlapping}: no proper suffix is equal to a prefix. For instance, the word $HHTHT$ is non-self-overlapping.

On the other hand, the largest expected wait time is achieved when $C_i = 1$ for all $i < \ell$, which is true for any word consisting entirely of $H$s or entirely of $T$s. For example, the expected wait time for $HHH$ is 14.

\begin{example}
With $(q,\ell) = (2,5)$, the words
\[HHHHT,\ HHHTT, \ HHTHT,\ HHTTT,\ HTHTT,\ HTTTT\]
all have autocorrelation vector $(1,0,0,0,0)$ and CLN of $2^{5-1} = 16$ .
\end{example}

\subsection{Odds for the game}

Going back to Penney's game, suppose Alice's and Bob's words are $w_1$ and $w_2$ respectively. Then from Theorem~\ref{thm:system} we have 
$$(1-qz)G(z) + G_{w_1}(z) + G_{w_2}(z) = 1$$
$$G(z) - z^{-\ell_1 }C_{w_1, w_1}(z) G_{w_1}(z) - z^{-\ell_2 }C_{w_2, w_1}(z) G_{w_2}(z) = 0$$
$$G(z) - z^{-\ell_1 }C_{w_1, w_2}(z) G_{w_1}(z) - z^{-\ell_2 }C_{w_2, w_2}(z) G_{w_2}(z)  = 0.$$
The probability that Alice wins the game is the same as the odds that $w_1$ appears before $w_2$ in a randomly generated string, which is
\[\frac{G_{w_1}(\frac1q)}{G_{w_2}(\frac1q)}.\]
This probability equals
$$\frac{q^{\ell_2}(C_{w_2,w_2}(\frac1q) - C_{w_2,w_1}(\frac1q))}{q^{\ell_1}(C_{w_1,w_1}(\frac1q) - C_{w_1,w_2}(\frac1q))},$$
Or, in terms of Conway leading numbers:
\[\frac{w_2\L w_2 - w_2\L w_1}{w_1\L w_1 - w_1\L w_2}.\]
In classical literature, this is known as \textbf{Conway's formula} for the two-player Penney's game \cite{BCG}.

\begin{example}
Suppose Alice selects $HTHT$, which has an expected wait time of 20 flips. Moreover, let Bob select $THTT$, which has an expected wait time of 18 flips. Surprisingly, despite the fact that Alice's wait time is longer, she wins with probability $\frac{9}{14}$.
\end{example}

The probability that the game takes $n$ flips is exactly $\frac{1}{q^n} (T_{w_1}(n) + T_{w_2}(n))$, and thus the expected length of the game is 
\[\sum_{n=0}^{\infty} \frac{n}{q^n} (T_{w_1}(n) + T_{w_2}(n)) = \frac{G'_{w_1}(\frac1q) + G'_{w_2}(\frac1q)}{q}.\]
In terms of Conway leading numbers,  the expected length of the game is
\[q \cdot \frac{(w_1 \L w_1)(w_2 \L w_2) - (w_1 \L w_2)(w_2 \L w_1)}{(w_1\L w_1 + w_2\L w_2) - (w_1\L w_2 + w_2 \L w_1)}.\]

\subsection{Optimal strategy for Bob}

Suppose Alice picks a word $w_1$, and Bob wants to pick a word $w_2$ to maximize his odds of winning the game. His odds of winning are
\[\frac{w_1 \L w_1 - w_1 \L w_2}{w_2 \L w_2 - w_2 \L w_1},\]
so his best beater would be a word that makes $w_1 \L w_2$ relatively small and $w_2 \L w_1$ relatively big compared to the other CLNs. As shown in \cite{F, GO}, a word $w_2$ of the form $w^* w_1(1) w_1(2) \dots w_1(\ell-1)$, i.e. a word for which $w_2(2,\ell) = w_1(1,\ell-1)$ fits the bill quite nicely. In fact, the following theorem is proven in \cite{GO}.

\begin{theorem}
Bob's best strategy is to pick a word $w_2$ for which $w_2(2,\ell) = w_1(1,\ell-1)$. In fact, this strategy always gives him odds $> 1$ of winning; these odds approach $q/(q-1)$ as $\ell \to \infty$.
\end{theorem}

The proof of this theorem relies heavily on Corollary \ref{cor:period} on periods. The exact choice of letter to pick for $w_2(1)$ is determined in \cite{F}.

\begin{example}
If Alice picks the word $HTHTH$, Bob's best strategy is to pick the word $HHTHT$. This gives him a $7:2$ odds of winning.
\end{example}

\section{Patterns in words and group action}\label{sec:piwga}

\subsection{A motivating example}

In the classical game, a word is generated by a sequence of letters. A natural extension of a word is a \emph{pattern}, where we identify a group of similar words with a single string of characters.

Explicitly, for the case $q=2$ we may identify a word composed of $H$'s and $T$'s with its \emph{conjugate}, or the result of replacing $H$'s with $T$'s and vice versa. Alice can choose a pattern for her word. For example, she can decide that all three characters are the same, effectively choosing two words $HHH$ and $TTT$. Bob can choose a pattern where the characters alternate, essentially picking $HTH$ and $THT$. That means if the game proceeds $HH\underline{HTH}$, then Bob wins. 

We represent the fact that Alice wants all characters the same as a pattern $aaa$, where $a$ could be either $H$ or $T$. In other words, to identify both a word and its conjugate collectively, we take the word beginning with $H$, replace all $H$'s with lowercase $a$'s and $T$'s with lowercase $b$'s.

\begin{example}
The pattern $aaa$ represents the two words consisting of the same letter: $HHH$ and $TTT$. The pattern $aba$ represents two words with alternating letters: $HTH$ and $THT$.
\end{example}

We can play Penney's game with patterns. Suppose Alice picks pattern $aaa$, and Bob picks pattern $aba$. Then they flip a coin. If three of the same flips in a row appear first, Alice wins. If three alternating flips in a row appear first, then Bob wins.

\subsection{Group action}

We can also extend patterns in an alphabet with two letters, $H$ and $T$, to patterns in larger alphabets. Let $\mathcal{A}$ be an alphabet of size $q$. We assume that the letters in the alphabet are: $A$, $B$, $C$, and so on.

Consider a subgroup $G \subseteq S_q$, where $S_q$ is a permutation group on $q$ elements. Group $G$ acts on the alphabet; formally, we consider the corresponding group action $\varphi : G \times \mathcal{A} \to \mathcal{A}$. For shorthand, we denote $\varphi(g,x) = g \cdot x$. Elements of the group send letters to letters, and thus words to words. The order of the group $G$ is denoted as $|G|$.

We use $G\cdot w$ to denote the orbit of the word $w$ under the action of group $G$. Two words $v$ and $w$ are \textbf{equivalent} if they belong to the same orbit of the group, notated $v \sim w$. These orbits split the set of words into equivalence classes.

\begin{example}
Consider the group action $S_3$ on three letters $\{A,B,C\}$ which permutes the three letters. Then $S_3 \cdot ABC = \{ABC, ACB, BAC, BCA, CAB, CBA\}$, while $S_3 \cdot AAA = \{AAA,BBB, CCC\}$.
\end{example}

 Within each orbit $G \cdot w$, we select a canonical representative. By convention, we choose it to be the lexicographically earliest word in this class. We denote the canonical representative as $s(G \cdot w)$. Thus, we have $G\cdot s(G \cdot w) = G \cdot w$. We call $s(G \cdot w)$ a \textbf{pattern}. To distinguish between words and patterns, we use lowercase letters for patterns. The canonical representative on a pattern is a word, so we can use the same operation on patterns as on words. For example, when we take the last letter of a pattern, we assume that we take the last letter of its canonical representative.

\begin{example}
We have $s(S_3 \cdot BCB) = aba$.
\end{example}

\begin{example}
Under the group action $G = S_3$ which permutes all letters, the 3-letter patterns are divided into 5 orbits with the corresponding canonical representatives: $aaa$, $aab$, $aba$, $abb$, $abc$.
\end{example}

\subsection{Other groups}

In addition to the symmetric group $S_q$, we consider a cyclic subgroup $\mathbb{Z}_q$, which cycles all the letters of the alphabet.

\begin{example}
Under $G= \mathbb{Z}_3$, the orbit of $BCA$ is $\{ABC, BCA, CAB\}$ with the canonical representative as $abc$. Note that $BCA \not \sim BAC$.
\end{example}

\begin{example}
Under $G= \mathbb{Z}_3$, the 9 three-letter patterns of length $3$ are $aaa$, $aab$, $aac$, $aba$, $abb$, $abc$, $aca$, $acb$, $acc$.
\end{example}

Here are some other interesting examples of groups. 

\begin{itemize}
\item We can shuffle vowels and consonants separately. In this case, our group is a product of two symmetric groups.
\item Our alphabet might be a deck of cards. We can permute cards while keeping the color, suit, or value in place. In this case, our group is a product of several symmetric groups.
\item We can allow permuting only the vowels while keeping the consonants in place. Our group is a symmetric group that is a subgroup of $S_q$.
\item We can choose the alternating group as a subgroup.
\item We can choose one non-trivial element $g$ in our group that reverses the order of the alphabet: namely, we have $g \cdot a(i) = a(q+1-i)$. In this case our group is $G= \mathbb{Z}_2$.
\end{itemize}

\section{Group action and correlation}\label{sec:GroupActionAndCorrelation}

Our first proposition shows that the correlations on words are invariant with respect to the group action $\varphi$.

\begin{proposition}
\label{prop:invar}
For any two words $v$, $w$ and $g \in G$, we have $C(v, w) = C(g\cdot v, g \cdot w)$.
\end{proposition}

\begin{proof}
Consider the $i$-th bit of $C(v,w)$, where the length of $v$ is $\ell$. It is equal to 1 if and only if $v(\ell+1-i,\ell) = w(1,i)$. Because $g$ is bijective on letters, it also must be bijective on words, so $(G\cdot v)(\ell+1-i,\ell) = (G \cdot w)(1,i)$. Thus, all bits are equal, and the two correlations are identical.
\end{proof}

Take any orbit of equivalent words $V = G \cdot v$. The sum of the correlations of the word $w$ with all words in orbit $V$ is the same for all words in the orbit of $w$.

\begin{corollary}
\label{cor:equiv}
For any two equivalent words $w_1 \sim w_2$, we have
$$\sum_{v_i \in V} C(w_1, v) = \sum_{v_i \in V} C(w_2, v_i) \quad \text{and} \quad \sum_{v_i \in V} C(v_i, w_1) = \sum_{v_i \in V} C(v_i, w_2).$$
\end{corollary}

\begin{proof}
By definition, there is some $g \in G$ such that $g \cdot w_1 = w_2$. As $V$ is a complete orbit, we have that $g \cdot V = V$. In other words,
$$\sum_{v_i \in V} C(w_1, v_i) = \sum_{v_i \in V} C(g \cdot w_1, g \cdot v_i) = \sum_{v_i \in V} C(w_2, v_i),$$
so the two sums are equal as desired. The second claim follows similarly.
\end{proof}

\subsection{The weight of a substring in a word}

The \textbf{stabilizer} $G_w$ of a word $w$ is the set $\{g \in G \mid g \cdot w = w\}$. Note that two words in the same orbit have the same stabilizers. Thus, the stabilizer of a pattern $p$ is well-defined; denote this stabilizer as $G_p$.
\begin{definition}
Given a word $w$ and its substring $w(i,j)$, the 
\textbf{weight of the substring in the word $w$} is the number of words $v$ in the orbit $G\cdot w$ such that $v(i,j) = w(i,j)$.
\end{definition}

\begin{example}
Consider word $BBC$ in an alphabet with 3 letters and group $S_3$, then the corresponding orbit is $\{AAB, AAC, BBA, BBC, CCA, CCB\}$. Hence, the weight of the suffix $C$ is 2 as there are two words $BBC$ and $AAC$ in the orbit that have suffix $C$.
\end{example}

The number of words in the orbit of a word $w$ is $|G/G_w|$. If $x$ is a substring in $w$, then $G_w$ is a subgroup of $G_x$. Thus, the number of words in the orbit of $w$ that have $x$ as a substring in the same place is $|G/G_x|$. Therefore, the definition of weight is equivalent to the following: The weight of the substring $x$ in the word $w$ equals:
\[\frac{|G_x|}{|G_w|}.\]

\begin{example}
If the group is $S_q$, then the stabilizer of a word $w$ is a subgroup $S_j$, where $j$ is the number of letters in the alphabet not used in $w$. Thus, the weight of substring $x$ in a word is $k!/j!$, where $k$ is the number of letters not used in $x$.
\end{example}

\begin{example}
If the group is $\mathbb{Z}_q$, then the stabilizer of a word $w$ is the identity. Thus, the weight of substring $x$ in a word is $1$.
\end{example}

\subsection{Correlation polynomials and Conway leading number for patterns}

Now we define the correlation polynomial between two patterns $p$ and $p'$ that do not have to be the same length. Let $\ell$ be the length of $p$.

\begin{definition}
The \textbf{correlation vector} between two patterns $p$ and $p'$ is denoted $\mathcal{C}(p,p')=(\mathcal{C}_{0},\dots ,\mathcal{C}_{\ell-1})$, where $\mathcal{C}_{i}$ is defined as follows.
\begin{itemize}
    \item If the suffix $x$ of $p$ of length $\ell-i$ is equivalent to the prefix of $p'$ of length $\ell-i$, then $\mathcal{C}_i$ is the weight of the suffix $x$ in the word $p$,
    \item If they are not equivalent, then $\mathcal{C}_i = 0$.
\end{itemize}
\end{definition}

This definition shares a few similarities with the correlation vector on words. The first entry $\mathcal{C}_0(p,p')$ is $1$ if $p \sim p'$ and is $0$ otherwise. Moreover, all entries of $\mathcal{C}(p,p')$ are integers.

\begin{definition}
The \textbf{correlation polynomial} of $p$ and $p'$ is defined as $\mathcal{C}_{p,p'}(z)=\mathcal{C}_{0}z^{0}+\dots +\mathcal{C}_{\ell-1}z^{\ell-1}$. It is a polynomial of degree at most $\ell-1$.
\end{definition}

Similar to before, the Conway leading number between two patterns is defined as the value of the correlation vector interpreted as a string in base $q$ and it is denoted as $p\mathcal{L} p'$.

\begin{definition}
The \textbf{Conway leading number (CLN)} for two patterns $p$ and $p'$ is 
\[p\mathcal{L} p' = q^{\ell-1} \mathcal{C}_{p,p'}\left(\frac{1}{q}\right).\]
\end{definition}

We provide another definition of correlation with the following proposition.

\begin{proposition}\label{prop:corr}
Pick any word $v$ belonging to the orbit represented by the pattern $p$, and $w$ belonging to the orbit represented by $p'$. Then
\[\mathcal{C}(p,p') = \sum_{v_i \in G\cdot v} C(v_i, w).\]
\end{proposition}

\begin{proof}
First notice that 
\[\sum_{v_i \in G\cdot v} C(v_i, w) = \frac{1}{|G_v|}\sum_{g \in G} C(g \cdot v, w).\]
Focus on a single $j$-th entry of the formula above. This entry is equal to the number of elements $g \in G$ for which the suffix ${g\cdot v}(j+1,\ell)$ of length $\ell-j$ matches the prefix $w(1,\ell-j)$. If these two words are not equivalent, then the count is $0$. If the two words are equivalent, the count is equal to the order of the stabilizer of $v(j+1,\ell)$. Thus, the $j$-th entry of the equation above equals $\mathcal{C}_j$ in the definition of the correlation as desired.
\end{proof}

The proposition above shows another way to define the correlation between patterns. Note that we cannot swap $v$ and $w$ in this definition. This is due to the following formula
\[|G_v|\sum_{v_i \in G\cdot v} C(v_i, w) = |G_w|\sum_{w_i \in G\cdot w} C(v, w_i),\]
which is true since after multiplying each vector sum by the constant preceding it, the $i$-th entry are either $0$ or are both the order of the stabilizer of $w(1,\ell-i)$.

\begin{example}
Consider the autocorrelation vector $\mathcal{C}(abc, abc)$ with respect to the group $S_3$. The weight of the suffix $abc$ in $abc$ is 1. The weight of the suffix $bc$ in $abc$ is 1, the weight of the suffix $c$ in $abc$ is 2. Thus the correlation vector is $(1,1,2)$. Another calculation can be done using Proposition~\ref{prop:corr}. The autocorrelation is
\begin{multline*}
C(ABC,ABC) + C(ACB,ABC) + C(BAC,ABC) + \\
C(BCA,ABC) + C(CAB,ABC) + C(CBA,ABC)
\end{multline*}
which evaluates to $(1,0,0) + (0,0,0) + (0,0,0) + (0,0,1) + (0,1,0) + (0,0,1) = (1,1,2)$. The corresponding correlation polynomial is thus $\mathcal{C}_{abc,abc}(z) = 1+z+2z^2$.
\end{example}

\subsection{Correlation vector}

Generalizing the definition of period for words, we say that a pattern $p = p(1)p(2) \dots p(\ell)$ has period $i$ if the $i$-th entry $\mathcal{C}_i$ of the vector $\mathcal{C}(p,p)$ is nonzero, i.e. if
\[p(1)p(2)\dots p(\ell-i) \sim p(i+1)p(i+2) \dots p(\ell),\]
or, equivalently, there exists $g \in G$ such that $g \cdot p(1,\ell-i) = p(i+1,\ell)$.

Note that if $p$ has period $s$ when the equivalence is realized by the group element $g$, then because $g \cdot (g \cdot p_i) = p_{i+2s}$ for $i+2s \le \ell$, we see that $p$ also has period $2s$. In this fashion, we get the following proposition in the fashion of Section~\ref{sec:originalproblem}.

\begin{proposition}
If a pattern has period $s$, it also has period $ks$ for all $k \geq 1$. More generally, if a pattern has period $s$ and period $t$, it also has period $s+t$.
\end{proposition}

Like in Section~\ref{sec:originalproblem}, the reverse implication is not true; if a pattern has periods $s < t$, it does not necessarily have period $t-s$.

\begin{example}
The pattern $p = abcdaec$  with symmetric group $G=S_3$ has $\mathcal{C}(p,p) = (1,0,1,1,2,6,24)$. In particular, it has periods $2$, $3$, $4$, $5$, and $6$, but not period $1$.
\end{example}

The reason that the above example does not satisfy the reverse implication is that the elements
\[g_2 = \begin{pmatrix}
a & b & c & d & e \\
c & d & a & e & b
\end{pmatrix}, \quad g_3 = \begin{pmatrix}
a & b & c & d & e \\
d & a & e & c & b
\end{pmatrix}\]
are the unique elements in $G$ for which $g_2 \cdot p(1,5) = p(3,7)$ and $g_3 \cdot p(1,4) = p(4,7)$. But they do not commute: $g_2g_3 \cdot p(3) = b \neq d = g_3g_2 \cdot p(3)$, meaning that $p$ cannot be extended to an eighth letter while having periods $2$ and $3$. If $g_2$ and $g_3$ commuted on every letter, then we would also have that $g_2^{-1}$ and $g_3$ commute, meaning $g_2^{-1}g_3 \cdot p(i) = p(i+1)$ for $i \leq \ell - 3$, and $g_3g_2^{-1} \cdot p(i) =  p(i+1)$ for $i \geq 3$. This would imply that $p$ has period $1$ for $\ell \geq 6$.

It thus makes sense that the implication is true for sufficiently long patterns since a large number of letters forces the respective group elements for each period to commute. In fact, we have the following lemma.

\begin{lemma}
\label{lem:subtract-per}
Let $p$ be a pattern with length $\ell \geq (q+1)s + t$ with periods $s < t$. Then $t-s$ is also a period of $p$.
\end{lemma}

To prove this, we need a quick lemma.

\begin{lemma}
\label{lem:all-letters}
Let $p$ be a pattern with length $\ell \geq qs$ and period $s$. Then each of the letters of $p$ appear somewhere in the first $qs$ letters of $p$.
\end{lemma}

\begin{proof}
By definition, there is an element $g \in G$ for which $g \cdot p(1,\ell-s) = p(s+1,\ell)$. Suppose the letter $x$ appears in $p$ and has index $k$: $x = p(k)$. Then the letter $y = p(k \mod s)$ must be in the same orbit as $x$ under $g$.
Consider the sequence
\[y, \quad g \cdot y, \quad g^2 \cdot y, \quad \dots, \quad g^{q-1} \cdot y.\]
Because the orbit has at most $q$ letters, this sequence contains the entire orbit of $y$. In particular, it must contain $x$. So there is some integer $i$ for which $x = g^i \cdot y$, which has index less than $qs$.
\end{proof}

We now prove Lemma~\ref{lem:subtract-per}.

\begin{proof}[Proof of Lemma~\ref{lem:subtract-per}]
Because $p$ has period $s$, there is an element $g_s \in G$ for which $g_s \cdot p(1,\ell-s) = p(s+1,\ell)$. Similarly, because $p$ has period $t$, there is an element $g_t \in G$ for which $g_t \cdot p(1,\ell-t) = p(t+1,\ell)$.

Thus, for all $1\leq i \leq qs$, we have $g_sg_t \cdot p(i) = p(i+s+t) = g_tg_s \cdot p(i)$. By Lemma~\ref{lem:all-letters}, the letter $p(i)$ ranges over all letters in $p$ as we vary $i$ in this range. So $g_s$ and $g_t$ commute for all letters in $p$. This implies $g^{-1}_s$ and $g_t$ also commute for all letters in $p$. In particular, we have
\begin{align*}
    g^{-1}_s g_t \cdot p(1,\ell+s-t) &= [g^{-1}_s g_t \cdot p(1,\ell-t)][g^{-1}_s g_t \cdot p(\ell-t+1,\ell-t+s)] \\
    &= [g^{-1}_s g_t \cdot p(1,\ell-t)][g_t g^{-1}_s \cdot p(\ell-t+1,\ell-t+s)] \\
    &= [p(-s+t+1,\ell-s)][p(\ell-s+1,\ell)] \\
    &= p(t-s+1,\ell),
\end{align*}
so $p$ has period $t-s$.
\end{proof}

In particular, from Lemma~\ref{lem:subtract-per} we have the following corollary.

\begin{corollary}
If $p$ has length $\ell$, \emph{least} period $s$ and period $t$ not divisible by $s$, then $t \geq \ell/(q+2) + 1$.
\end{corollary}

\begin{proof}
If $t$ is the least period not divisible by $s$, then $t-s$ cannot be a period. The contrapositive of Lemma~\ref{lem:subtract-per} implies that $\ell \leq (q+1)s + t - 1$. Since $s+1 \leq t$, we have $\ell \leq (q+2)t - q - 2$ or $t \geq \ell/(q+2) + 1$, as desired.
\end{proof}

\subsection{Non-self-overlapping patterns}

Recall that, for a word $w$ of length $\ell$, the vector $C(w,w) = (C_0, C_1, \dots, C_{\ell - 1})$ satisfies $C_0 = 1$ for $1 \leq i < \ell$, so all CLNs are at least $q^{\ell-1}$. This bound is exact. It is achieved when no proper suffix is equivalent to a prefix, that is for non-self-overlapping words.

We call a pattern \textbf{non-self-overlapping} if no proper prefix is equivalent to suffix. Such a pattern can only exist if the alphabet is not one orbit of group $G$, that is, if there are two letters that are not equivalent to each other. Indeed, the last letter and the first letter of a pattern should not be equivalent in a non-self-overlapping pattern.

Suppose letters $a$ and $b$ are not equivalent, then any non-self-overlapping word in these two letters provide a non-self-overlapping pattern.

\begin{example}
Suppose our group permutes vowels and consonants separately without mixing them. Then, a vowel is not equivalent to a consonant. Let us denote a vowel by $H$ and a consonant by $T$. Thus, any pattern in our alphabet can be mapped into a word in the two-letter alphabet. If the image of this mapping is a non-self-overlapping word, then the original pattern is non-self-overlapping.
\end{example}

We call a pattern \textbf{almost-non-self-overlapping} if no proper prefix is equivalent to a suffix, except for the prefix of length 1.

\begin{example}
Consider pattern $aaa\cdots aaab$. If $a$ and $b$ belong to different equivalent classes under the group action, then the pattern is non-self-overlapping. Otherwise, it is almost-non-self-overlapping.
\end{example}

\subsection{Lower bound for CLN}

The lower bound $q^{\ell - 1}$ for words is not always achievable for patterns. We strengthen the lower bound with the following claim.

\begin{theorem}
Consider the lowest CLN achieved by any pattern of length $\ell$.
\begin{itemize}
    \item[(i)] If an orbit of a single letter under action of group $G$ does not cover all the alphabet, then the lowest possible CLN is $q^{\ell-1}$.
    \item[(ii)] Otherwise, the lowest possible CLN is between $q^{\ell-1} + 1$ and $q^{\ell-1} + q - 1$ inclusive.
    \item[(iii)] The lowest bound $q^{\ell-1} + 1$ is achived for group $\mathbb{Z}_q$ and the lowest bound $q^{\ell-1} + q - 1$ is achieved for group $S_q$.
\end{itemize}
\end{theorem}

\begin{proof}
To begin, we know that $\mathcal{C}_0 = 1$, and so $p\mathcal{L}p \geq q^{\ell-1}$.

If an orbit of a single letter under the group action $\varphi$ of group $G$ does not cover all letters in the alphabet, then there are two letters $A$ and $B$ that belong to different orbits. As such, any non-self-overlapping word in the alphabet $\{A,B\}$ of length $\ell$ is in an orbit whose corresponding pattern has CLN equal to $q^{\ell - 1}$, proving (i).

If an orbit of a single letter under action of group $G$ does cover all the alphabet, then the last character of pattern $p$ is equivalent to the first character, or $\mathcal{C}_{\ell-1} \geq 1$ and therefore $p \mathcal{L} p \geq q^{\ell-1} + 1$.

Consider the pattern $p = aa\ldots ab$. The number of elements in the orbit of the word $s(p)$ that end in $B$ does not exceed $q-1$. It follows that $p \mathcal{L} p \leq q^{\ell-1} + q-1$. This pattern $p$ provides the exact values for the lower bound for groups $\mathbb{Z}_q$ and $S_q$.
\end{proof}

For words, the lower bound cannot be achieved if the first and last letters match. However, in the case of patterns, the first and last character \emph{can} match and still achieve the lower bound.

\begin{example}
For $(q,\ell) = (4,5)$ and the group $G =  S_4$, the lower CLN bound of $4^{5-1} + 4 - 1 = 259$ is achieved for the patterns $aaaab$, $aaaba$, $aabab$, $abaaa$, $ababb$, and $abbbb$.
\end{example}

Note that a word that achieves the lower bound for words may not correspond to a pattern that achieves the lower bound for patterns.

\begin{example}
For $q = 2$, the word $AABB$ has autocorrelation $(1,0,0,0)$, thus, achieving the lower bound for CLN for words. If we consider group $S_2$, then the pattern $aabb$ has autocorrelation $(1,0,q-1,q-1)$, which does not match the minimum CLN.
\end{example}

In general, there does not appear to be a simple rule to generate all patterns achieving the lower bound.

\section{Generating functions}\label{sec:gfs}

We now generalize an analogue of Theorem \ref{thm:system} to sets of patterns. A \textbf{reduced} set of patterns is a set where there are no patterns $p$ and $p'$, such that some substring of $p'$ defines a pattern equivalent to $p$. For instance, the set $\{aba,aabcb\}$ is not reduced, as $aba \sim bcb$. Note that if a set of patterns is reduced, the joint set of words formed by the union of their orbits is also reduced. The reverse direction does not hold: for instance, the set of words $\{ABA, AABCB\}$ is reduced while the corresponding set of patterns $\{aba, aabcb\}$ is not.

Fix a group $G$ to form our group action; let $\mathcal{S} = \{p_1, p_2, \dots, p_k\}$ denote a reduced set of $k$ patterns with lengths $\ell_1, \ell_2, \dots, \ell_k$, all composed of letters from an alphabet of size $q$. It is important that the set $\mathcal{S}$ is reduced for the same reasons as before.

Once more, define
\begin{itemize}
    \item $\mathcal{A}(n,\mathcal{S})$ to be the number of words (not patterns) of length $n$ not containing any subword represented by any pattern $p \in \mathcal{S}$; and
    \item $\mathcal{T}_{p_i}(n,\mathcal{S})$  to be the number of words (not patterns) of length $n$ not containing any subword represented by any pattern $p \in \mathcal{S}$, except for a single word represented by the pattern $p_i$ at the end of the word.
\end{itemize}
Then we set the generating functions
\[\mathcal{G}(z,\mathcal{S}) = \sum_{k=0}^\infty \mathcal{A}(k,\mathcal{S})z^k, \quad \mathcal{G}_{p_i}(z,\mathcal{S}) = \sum_{k=0}^\infty \mathcal{T}_{p_i}(k,\mathcal{S})z^k,\]
and sometimes drop $\mathcal{S}$
when it is clear which set $\mathcal{S}$ of patterns we are referring to. We verify a quick proposition about generating functions of equivalent words. Given a set of patterns $\mathcal{S}$, we denote by $S$ the set of words that corresponds to the union of the orbits of all the patterns.

\begin{proposition}
\label{invar-genfunc}
Take a set $S$ of words which is invariant under the group action $\varphi$ induced by the group $G$. Then for any equivalent words $v \sim w$ in $S$, we have $G_v(z,S) = G_w(z,S)$.
\end{proposition}

\begin{proof}
Take a $g \in G$ for which $g \cdot w = v$. Then for any first occurrence word $x$ containing $w$, the word $g \cdot x$ is a first occurrence word containing $v$. This word still avoids all other words in $S$, since $S$ is invariant under $G$. 

As $g$ is invertible, we have $\mathcal{T}_w(n,S) = \mathcal{T}_v(n,S)$ and the two generating functions are identical.
\end{proof}

Finally, let the orbits corresponding to the patterns $p_1, p_2, \dots, p_k$ have sizes $r_1, r_2, \dots, r_k$ respectively. We derive a system of equations in the manner of Theorem~\ref{thm:system}.

\begin{theorem}
\label{thm:system2}
The generating functions $\mathcal{G}(z),\mathcal{G}_{p_1}(z),\mathcal{G}_{p_2}(z), \dots, \mathcal{G}_{p_k}(z)$ satisfy the following system of equations:
$$(1-qz)\mathcal{G}(z) + \mathcal{G}_{p_1}(z) + \mathcal{G}_{p_2}(z) + \dots + \mathcal{G}_{p_k}(z) = 1$$
$$\mathcal{G}(z) - \frac1{r_1} z^{-\ell_1}\mathcal{C}_{p_1, p_1}(z) \mathcal{G}_{p_1}(z) - \dots - \frac1{r_k} z^{-\ell_k}\mathcal{C}_{p_k, p_1}(z)\mathcal{G}_{p_k}(z) = 0$$
$$\mathcal{G}(z) - \frac1{r_1} z^{-\ell_1}\mathcal{C}_{p_1, p_2}(z) \mathcal{G}_{p_1}(z) - \dots - \frac1{r_k} z^{-\ell_k}\mathcal{C}_{p_k, p_2}(z)\mathcal{G}_{p_k}(z) = 0$$
$$\vdots$$
$$\mathcal{G}(z) - \frac1{r_1}z^{-\ell_1}\mathcal{C}_{p_1, p_k}(z) \mathcal{G}_{p_1}(z) - \dots - \frac1{r_k}z^{-\ell_k}\mathcal{C}_{p_k, p_k}(z)\mathcal{G}_{p_k}(z) = 0.$$
\end{theorem}

\begin{proof}
As before, we denote the orbit of words represented by $p_i$ as $G \cdot p_i$. The key is to evaluate Theorem~\ref{thm:system} with the set of words
\[S = \bigcup_{i=1}^{k} \ (G \cdot p_i) = (G \cdot p_1) \cup (G \cdot p_2) \cup \dots \cup (G \cdot p_k),\]
or the set of words composed of the union of the orbits represented by $p_1, p_2, \dots, p_k$. These orbits are pairwise disjoint since $\mathcal{S}$ is reduced; in addition, this new set of words is also reduced. The first equation of Theorem~\ref{thm:system} becomes
\begin{align*}
1 &= (1-qz)G(z,S) + \sum_{w \in G\cdot p_1} G_{w}(z,S) + \dots + \sum_{w \in G\cdot p_k} G_{w}(z,S) \\
&= (1-qz)\mathcal{G}(z) + \mathcal{G}_{p_1}(z) + \dots + \mathcal{G}_{p_k}(z).
\end{align*}
In addition, for any word $v \in G \cdot p_j$, we have $\sum_{w\in G \cdot p_i} C_{w,v}(z) = \mathcal{C}_{p_i,p_j}(z)$
by our alternative definition of correlation. Moreover,
$$\mathcal{G}_{p_j}(z,\mathcal{S}) = \sum_{w \in G \cdot p_j}G_{w}(z,S) = r_j G_{v}(z,S)$$
by Proposition~\ref{invar-genfunc}. Thus, the equation from Theorem~\ref{thm:system} corresponding to a word $w \in G \cdot p_i$ becomes
$$\mathcal{G}(z) - \frac1{r_1} z^{-\ell_1} \mathcal{C}_{p_1, p_i}(z) \mathcal{G}_{p_1}(z) - \dots - \frac1{r_k} z^{-\ell_k} \mathcal{C}_{p_k, p_i}(z) \mathcal{G}_{p_k}(z) = 0,$$
as desired. 
\end{proof}

\subsection{Normalized correlation for patterns}

Note the striking similarity between Theorem~\ref{thm:system} and Theorem~\ref{thm:system2}. Namely, we can get from Theorem~\ref{thm:system} to the other by replacing the correlation polynomial for $C_{w_i,w_j}(z)$  with $\frac{1}{r_i} \mathcal{C}_{p_i, p_j}(z)$. To formalize this, we define the normalized correlation vector for patterns.

\begin{definition}
Let $p$ and $p'$ be patterns, and let $p$ represent an orbit of size $r$. The \textbf{normalized correlation vector} $\mathcal{C}^*(p,p')$ is equal to $\frac{1}{r} \mathcal{C}(p,p')$. Similarly, the \textbf{normalized correlation polynomial} $\mathcal{C}^*_{p,p'}(z)$ is equal to $\frac{1}{r} \mathcal{C}_{p,p'}(z)$.
\end{definition}

While we default to the previous definition of correlation, normalized correlation has many nice properties. Let $p$ have length $\ell$. First, by definition the $i$-th entry of $\mathcal{C}_i^*(p,p')$ is zero if the prefix $p(1,\ell-i)$ and the suffix $p'(i+1,\ell)$ are not equivalent, and $\frac1r |G_{p(1,\ell-i)}|/|G_p| = |G_{p(1,\ell-i)}|/|G|$ if the prefix and suffix are equivalent. The $i$-th entry can be interpreted as the fraction of the elements of $G$ which fix $p(1,\ell-i)$, which is a rational number between $0$ and $1$.

We can also think of normalized correlation as an average of $|G|$ correlations. Specifically, pick a word $v$ from the orbit represented by $p$, and $w$ from the orbit represented by $p'$. Using Proposition~\ref{prop:corr} we can write
\[\mathcal{C}^*(p,p') = \frac{1}{|G|} \sum_{g \in G} C(g \cdot v, w),\]
which is the average of the $|G|$ correlations consisting of a varying word from the orbit represented by $p$ and a fixed word from the orbit represented by $p'$. 

Finally, normalized correlation is fixed regardless of whether we choose to vary the word associated to $p$ and fix the period associated to $p'$, or vice versa. Explicitly, we have
\[\mathcal{C}^*(p,p') = \frac{1}{|G|} \sum_{g \in G} C(g \cdot v, w) = \frac{1}{|G|} \sum_{g \in G} C(v, g\cdot w).\]

To derive the pattern analogues of various results for words, we obey the following principle.

\begin{principle}
For any result on words that is derived from Theorem~\ref{thm:system} (such as generating functions, odds, etc.), we replace any word correlation polynomial $C_{v,w}(z)$ with the normalized pattern correlation polynomial $\mathcal{C}^*_{p,p'}(z)$ to get the corresponding result derived from Theorem~\ref{thm:system2}.
\end{principle}

For instance, the following corollary is immediate.

\begin{corollary}
For $k=1$ we get
$$\mathcal{G}(z) = \frac{\mathcal{C}^*_{p,p}(z)}{z^{\ell} + (1-qz)\mathcal{C}^*_{p,p}(z)} = \frac{\mathcal{C}_{p,p}(z)}{rz^{\ell} + (1-qz)\mathcal{C}_{p,p}(z)},$$
$$\mathcal{G}_p(z) = \frac{z^{\ell}}{z^{\ell} + (1-qz)\mathcal{C}^*_{p,p}(z)} = \frac{rz^{\ell}}{rz^{\ell} + (1-qz)\mathcal{C}_{p,p}(z)}.$$
\end{corollary}

We see that here, as with words, the denominator is the same for both functions.

\begin{example}
For the pattern $p=aa$ with $q=2$ and group $G = S_2$ we have $\mathcal{C}(p,p) = (1,1)$ and 
$$\mathcal{G}(z) = \frac{z+1}{1-z}, \quad \mathcal{G}_{p}(z) = \frac{2z^2}{1-z}.$$
Here, the coefficients of both generating functions are eventually constant.
\end{example}

\section{The expected wait time}\label{sec:ewt}

We can also compute the expected wait time for a pattern $p$ with orbit size $r$ to appear, given that each letter appears with equal probability $\frac1q$: this is just $q^\ell \mathcal{C}^*_{p,p}(\frac1q)=\frac1r q^{\ell}\mathcal{C}_{p,p}(\tfrac1q)$. In terms of the Conway leading number, the expected wait time is $\frac qr \cdot p\mathcal{L}p$.

\begin{example}
For $q =2$, the pattern $p = abab$ has $\mathcal{C}(p,p) = (1,1,1,1)$ and orbit size $2$, and thus $p\mathcal{L}p = 1111_2 = 15$. Therefore, the pattern has expected wait time $\frac{2}{2} \cdot (15) = 15$.
\end{example}

Recall that the expected wait time for a pattern $p$ is equal to $\frac{q}{r} \cdot p\mathcal{L} p = q^\ell \mathcal{C}^*_{p,p}(\frac1q)$. We may rewrite this value as
\[\frac{1}{|G|}\sum_{i=1}^\ell q^i |G_p| \mathcal{C}_i(p,p),\]
by the orbit-stabilizer theorem. Note that $|G_p| \mathcal{C}_i(p,p)$ is equal to $0$ if $p(1,\ell-i) \not\sim p_{i+1,\ell}$, and $|G_{p(1,\ell-i)}|$ otherwise.

\subsection{Bounds on the expected wait time}

We present two similar results, explicitly giving patterns that achieve the lowest and highest possible expected wait time.

\begin{proposition}
Let $x \in \mathcal{A}$ be the letter with the largest stabilizer. Then the greatest expected wait time of a pattern of length $\ell$ is \[\frac{|G_x|}{|G|}(q+q^2+\dots + q^{\ell}),\]
achieved by the pattern $p = xx\dots x$.
\end{proposition}

\begin{proof}
In fact, we make the stronger claim that the maximal value of each individual coefficient $|G_p|\mathcal{C}_i(p,p)$ of the expansion of $\frac{q}{r} p\mathcal{L}p$ is achieved at $xx\dots x$, using the same letter $x$ for each $i = 0,1,2,\dots, \ell-1$. This certainly implies the original claim. Assume $\mathcal{C}_i > 0$, so it remains to maximize $|G_p|\mathcal{C}_i(p,p) = |G_{p(1,\ell-i)}|$.

Suppose the prefix $p(1,\ell-i)$ contains the letters $a_1, a_2, \dots, a_k$. Then the size of the stabilizer of the prefix is at most the size of the stabilizer of $a_1$, so we may as well assume the prefix (and the pattern) consists of a single letter. Then we select the letter $x$ with the largest stabilizer (if there are ties, pick one arbitrarily) to comprise our pattern.
\end{proof}

\begin{proposition}
Let $\ell \geq q+1$ be an integer. If there are two letters within disjoint orbits under the group action by $G$, then the least expected wait time of a pattern of length $\ell$ is $q^\ell/|G|$. Otherwise, the least expected wait time is $q^{\ell}/|G| + 1$.
\end{proposition}

\begin{proof}
Consider an almost-self-non-overlapping pangramic
\[p = a_1a_2\dots a_{q-1} \underbrace{a_q \dots a_q}_{\ell+1-q \ a_q\text{'s}}.\]
We have 
\begin{itemize}
    \item[(1)] $|G_p| = 1$ as $p$ contains all $q$ letters of $\mathcal{A}$, and
    \item[(2)]  $\mathcal{C}_i = 0$ for all $1 \leq i \leq \ell-1$ as $p$ is almost-self-non-overlapping.
\end{itemize}

We first tackle the case where $G$ has two letters in disjoint orbits. Call these two letters $x$ and $y$. Without loss of generality we can assume that $a_1 = x$ and $a_q=y$. Then $\mathcal{C}_\ell = 0$.

Note that this pattern achieves the least possible value of $\frac{1}{r}\mathcal{C}_i$ for all $1 \leq i \leq \ell - 1$. Moreover, since $\frac{1}{r}\mathcal{C}_0 = |G_p|/|G|$ and $|G_p| = 1$ for a pattern $p$ containing all letters in $\mathcal{A}$, such a pattern also minimizes $\mathcal{C}_0$ as well. Thus, this pattern achieves the least possible expected wait time.

If every letter is in a single orbit, then we must have $\mathcal{C}_{\ell - 1} > 0$. The pattern $p$ still minimizes $\mathcal{C}_i$ for $0 \leq i \leq \ell - 2$. Moreover, note $\frac1r \mathcal{C}_{\ell-1} = |G_{p(1,1)}|/|G|$. Because all letters are in the same orbit, they all have a stabilizer of size $|G|/q$ by the orbit-stabilizer theorem. Thus, $\frac1r \mathcal{C}_{\ell - 1} = 1/q$ is a constant, and as such $p$ still minimizes every $\mathcal{C}_i$ and therefore the expected wait time. The claim now follows.
\end{proof}

\section{Odds}\label{sec:odds}

We can compute the winning probabilities for a game with patterns using Theorem \ref{thm:system2} and the Main Principle.

\begin{theorem}
Suppose Alice and Bob pick the patterns $p_1$ and $p_2$, with lengths $\ell_1, \ell_2$ and orbit sizes $r_1, r_2$ respectively. We assume that patterns are such that $\mathcal{S} = \{p_1, p_2\}$ is reduced. Then the odds of Alice winning the game are:
\[\frac{q^{\ell_2}(\mathcal{C}^*_{p_2,p_2}(\frac1q) - \mathcal{C}^*_{p_2,p_1}(\frac1q))}{q^{\ell_1}(\mathcal{C}^*_{p_1,p_1}(\frac1q) - \mathcal{C}^*_{p_1,p_2}(\frac1q))}
= \frac{r_1q^{\ell_2}(\mathcal{C}_{p_2,p_2}(z) - \mathcal{C}_{p_2,p_1}(z))}{r_2q^{\ell_1}(\mathcal{C}_{p_1,p_1}(z) - \mathcal{C}_{p_1,p_2}(z))}
= \frac{r_1}{r_2}\cdot \frac{p_2\mathcal{L}p_2 - p_2\mathcal{L}p_1}{p_1\mathcal{L}p_1 - p_1\mathcal{L}p_2}.\]
\end{theorem}

We see that to get to the formula for patterns from the formula for words in calculating the odds, we need to replace the correlation between words with the correlation between patterns and then multiply the result by $\frac{r_1}{r_2}$.

Given the adjusted odds, we can adjust the expected length of the game.
\begin{corollary}
The expected length of the game is 
\[q\cdot \frac{(p_1 \mathcal{L} p_1)(p_2 \mathcal{L} p_2) - (p_1 \mathcal{L} p_2)(p_2 \mathcal{L} p_1)}{r_1(p_2 \mathcal{L} p_2 - p_2 \mathcal{L} p_1) + r_2(p_1 \mathcal{L} p_1 - p_1 \mathcal{L} p_2)}.\]
\end{corollary}

\subsection{Optimal strategy for Bob}

Recall from Section~\ref{sec:originalproblem} that for the game on words, Bob has a winning strategy. Specifically, if Alice picks the word $w(1)w(2)\dots w(\ell)$ the best choice for Bob is the word $w^*w(1)w(2)\dots w(\ell)$ for some $w^*$ which makes his odds of winning greater than $1$. In this section, we show the following pattern analogue of the strategy for words.

\begin{theorem}
\label{thm:pat-bob-wins}
Fix $q$ and let $\ell$ be sufficiently large. If Alice picks the pattern $p_1 = p_1(1)p_1(2)\dots p_1(\ell)$, then Bob's best beater is a pattern $p_2$ for which $p_2(2,\ell) = p_2(1,\ell-1)$. Bob can always choose such a word such that his odds of winning exceed $1$. Moreover, as $\ell \to \infty$ these winning odds approach $|G_{p_1}|/|G_{p_2}| \cdot q/(q-1)$.
\end{theorem}

Note that, in this case, we prove Bob wins for sufficiently large $\ell$. In Section~\ref{sec:symmetric}, we provide a family of examples with small $\ell$ where Alice actually has a winning strategy.

To prove this, we generalize the methods used in~\cite{GO}, which showed for words satisfying $w_2(2,\ell) = w_1(1,\ell-1)$, the quantity $w_1 \L w_2$ is relatively negligible while $w_2 \L w_1$ is relatively large, causing Bob's odds $\frac{w_1 \L w_1 - w_1 \L w_2}{w_2 \L w_2 - w_2 \L w_1}$ to be large. This all stems from the results on periods in Section~\ref{sec:originalproblem}, which we generalized in Section~\ref{sec:GroupActionAndCorrelation}. We now use those generalizations to prove Theorem~\ref{thm:pat-bob-wins}.

Beforehand, we prove a lemma, now showing that Bob can pick a pattern satisfying certain conditions.
\begin{lemma}
\label{lem:bob-pick}
If Alice chooses the pattern $p_1 = p_1(1)p_1(2) \dots p_1(\ell)$, Bob can pick a pattern $p_2 = p_2(1) p_1(1) \dots p_1(\ell-1)$ such that:
\begin{itemize}
    \item If $p_2$ has a period $t$, then $t \geq \ell/(q+2) + 1$.
    \item We have $|G_{p_2}| \leq |G_{p_1}|$.
\end{itemize}
\end{lemma}

\begin{proof}
Note that if $t$ is a period of $p_2$, then $t$ is also a period of $p_1$. We let $s$ be the least nontrivial period of $p_2$.

If $s \geq \ell/(q+2) + 1$, then we are trivially done, so assume $s \leq \ell/(q+2) + 1$. Then by Lemma~\ref{lem:all-letters} the letter $p_1(\ell)$ appears in the first $qs \leq \ell - 1$ letters of $p_2$, namely $p_1(1,\ell-1)$, so it appears in $p_2$, implying $p_2$ contains all of the letters in $p_1$ and so $|G_{p_2}| \leq |G_{p_1}|$.
    
It suffices to pick $p_2$ such that $ks$ is not a period of $p_2$ for any $k$: if we do, the smallest possible period of $p_2$ is the smallest period of $p_1$ not divisible by $s$, which we know to be at least $\ell/(q+2) + 1$. Note that either $p_1(s)$ and $p_1(s+1)$ are the same letter, or different letters. If they are the same, then $p_1(ks) = p_1(ks+1)$ as well, and picking $p_2(1) \neq p_1(1)$ is enough: that way,
\[
p_2(1)p_2(2) = p_2(1)p_1(1) \not\sim p_1(ks)p_1(ks+1) = p_2(ks+1)p_2(ks+2),
\]
since the former has two different letters and the latter has two of the same letter, and so $p_2$ does not have period $ks$ for any $k$. Similarly, if $p_1(s) \neq p_1(s+1)$, then pick $p_2(1) = p_1(1)$.
\end{proof}

We are now ready to prove Theorem~\ref{thm:pat-bob-wins}.

\begin{proof}[Proof of Theorem~\ref{thm:pat-bob-wins}]
Let Bob pick a pattern that satisfies the two conditions in Lemma~\ref{lem:bob-pick}. We claim Bob has winning odds for sufficiently large $\ell$, i.e. there is an $M$ for which $\ell > M$ implies
\[\frac{\frac{1}{r_1}p_1\mathcal{L}p_1 - \frac{1}{r_1}p_1\mathcal{L}p_2}{\frac{1}{r_2}p_2 \mathcal{L} p_2 - \frac{1}{r_2} p_2 \mathcal{L} p_1} > 1.\]
To do this, we bound each of the terms
\[\frac{1}{r_m} p_m \mathcal{L} p_n = \mathcal{C}^*_0(p_m,p_n) q^{\ell-1} + \dots + \mathcal{C}^*_{\ell-1}(p_m,p_n).\]
Note that all normalized coefficients are between $0$ and $1$.

First, because $\mathcal{C}^*_{t+1}(p_1,p_2) > 0$ if and only if $p_2$ has period $t$, i.e. $t \geq \lceil \ell/(q+2) \rceil + 1$, all of the coefficients $\mathcal{C}^*_0(p_1,p_2), \dots, \mathcal{C}^*_{\lceil \ell/(q+2) \rceil + 1}(p_1,p_2)$ are zero, and since all other coefficients are at most $1$ we have the bound
\begin{align*}
\frac{1}{r_1} p_1 \mathcal{L} p_2 \leq q^{\lfloor (q+1)\ell/(q+2) \rfloor - 3} + q^{\lfloor (q+1)\ell/(q+2) \rfloor - 4} + \dots + 1 = O(q^{(q+1)\ell/(q+2)}).
\end{align*}
Similarly, since $\mathcal{C}^*_t(p_2,p_2) = \mathcal{C}^*_{t+1}(p_2,p_1)$, both of which are nonzero if and only if $t$ is a period of $p_2$, we may bound the denominator:
\begin{align*}
    \frac{1}{r_2} p_2 \mathcal{L} p_2 &- \frac{1}{r_2} p_2 \mathcal{L} p_1 \\
    &= \sum_{i=0}^{\ell-1} \mathcal{C}^*_i(p_2,p_2) q^{\ell-1-i} - \sum_{i=0}^{\ell-1} \mathcal{C}^*_i(p_2,p_1) q^{\ell-1-i} \\
    &\leq 1 + \sum_{i=0}^{\ell-2} (\mathcal{C}^*_i(p_2,p_2) q^{\ell-1-i} - \mathcal{C}^*_{i+1}(p_2,p_1) q^{\ell-2-i}) \\
    &= 1 + \sum_{i=0}^{\ell-2} \mathcal{C}^*_i(p_2,p_2) (q^{\ell-1-i} - q^{\ell-2-i}) \\
    &= 1 + \frac{|G_{p_2}|}{|G|}(q^{\ell-1} - q^{\ell-2}) + \sum_{i=\lceil \frac{\ell}{(q+2)} \rceil + 1}^{\ell-1} \mathcal{C}^*_i(p_2,p_2) (q^{\ell-1-i} - q^{\ell-2-i}) \\
    &\leq 1 + \frac{|G_{p_2}|}{|G|}(q^{\ell-1} - q^{\ell-2}) + \sum_{i=\lceil \frac{\ell}{(q+2)} \rceil + 1}^{\ell-1} (q^{\ell-1-i} - q^{\ell-2-i}) \\
    &= \frac{|G_{p_2}|}{|G|}(q^{\ell-1} - q^{\ell-2}) + q^{\lfloor (q+1)\ell/(q+2) \rfloor - 2}.
\end{align*}
So Bob's odds of winning are bounded below by
\[\frac{\frac{1}{r_1} p_1 \mathcal{L} p_1 - O(q^{(q+1)\ell/(q+2)})}{\frac{|G_{p_2}|}{|G|}(q^{\ell-1} - q^{\ell-2}) + q^{\lfloor (q+1)\ell/(q+2) \rfloor - 2}} \geq \frac{\frac{|G_{p_1}|}{|G|}q^{\ell-1} - O(q^{(q+1)\ell/(q+2)})}{\frac{|G_{p_2}|}{|G|}(q^{\ell-1} - q^{\ell-2}) + q^{\lfloor (q+1)\ell/(q+2) \rfloor - 2}},\]
which is greater than $1$ as $\ell$ grows large. It actually approaches $|G_{p_1}|/|G_{p_2}| \cdot q/(q-1) + O(q^{-\ell/(q+2)})$.

To show that a pattern of this form is the best beater, suppose Bob picks $p_2$ for which $p_2(2,\ell) \not\sim p_1(1,\ell-1)$. Then using the trivial bounds $\frac{1}{r_1} p_1 \mathcal{L} p_2 \geq 0$ and
\[\frac{1}{r_2} p_2 \mathcal{L} p_2 - \frac{1}{r_2} p_2 \mathcal{L} p_1 \leq q^{\ell-1} - \frac{q^{\ell-2}-1}{q-1},\]
the odds the Bob wins are now at most
\[\frac{\frac{1}{r_1} p_1 \mathcal{L} p_1}{q^{\ell-1} - \frac{q^{\ell-2}-1}{q-1}}.\]
So it suffices to show
\[\frac{\frac{1}{r_1} p_1 \mathcal{L} p_1 - O(q^{(q+1)\ell/(q+2)})}{(q^{\ell-1} - q^{\ell-2}) + q^{\lfloor (q+1)\ell/(q+2) \rfloor - 2}} \geq \frac{\frac{1}{r_1} p_1 \mathcal{L} p_1}{q^{\ell-1} - \frac{q^{\ell-2}-1}{q-1}}.\]
This can be shown to be true for $q\geq 3$ and sufficiently large $\ell$, after a tedious but trivial computation, which we omit. We later discuss the case $q=2$ in Section~\ref{sec:cyclic}, where we show this case is equivalent to the $q=2$ case on words.
\end{proof}

We now discuss the results for two specific groups $G$: the cyclic group $\mathbb{Z}_q$, and the symmetric group $S_q$.

\section{Cyclic group}\label{sec:cyclic}

For this section, we assume that the letters in the alphabet have an assigned order. So we can number the letters with the residues modulo $q$: $\{0,1,\ldots,q-1\}$.

We consider the group action under the cyclic group $G = \mathbb{Z}_q = \{0, 1, \dots, q-1\}$. An element $g \in G$ shifts each letter forward by $g$, wrapping around if necessary. In other words, the action of $g$ on a letter numbered $i$ is the letter numbered $g + i \pmod{q}$. Here, to comply with tradition, we use the plus sign for the action of this group. This group action on words is known as a \emph{Caesar shift}. 

\begin{example}
With $q=26$ and $\mathcal{A}$ the English alphabet in its canonical order, we have that $6 + {FUSION} = {LAYOUT}$.
\end{example}

Given a pattern $p$ there is exactly one word in the orbit of the pattern starting with a given letter; thus, there are exactly $q$ elements in the orbit it represents. We describe each pattern with its lexicographically earliest element in its orbit.

\begin{example}
The word ${LAYOUT}$ belongs to the orbit labeled with the pattern $15 + {LAYOUT}$ or ${apndji}$, and the word ${BOOKKEEPER}$ belongs to the orbit labeled $25 + {BOOKKEEPER}$ or ${annjjddodq}$.
\end{example}

As the letters are numbered, we may ``subtract" one letter from another. Note that the difference between any two letters is invariant with regard to shifting by an element $g \in G$, namely $a - b \equiv (a+g) - (b+g) \pmod{q}$. Thus, a word is uniquely determined by its first letter and the differences between each pair of consecutive letters; a pattern is uniquely determined by just the differences.

We formalize this idea as follows. Given a pattern $p = p(1)p(2) \dots p(\ell)$ of length $\ell \geq 2$, we define the \emph{adjacency signature} $S(p)$ to be a word of length $\ell-1$, consisting of the letters corresponding to the integers $\{s(1), s(2), \dots, s(\ell-1)\}$, where $s(i) \in \{0,1,\dots, q-1\}$ is the unique integer $k$ for which $k + p(i) = p(i+1)$. An integer in the set $\{s(1), s(2), \dots, s(\ell-1)\}$ corresponds to a unique letter in our alphabet, so we interchangeably use the corresponding letters to the word $S(p)$.

\begin{example}
The pattern $p = {apndji}$ has adjacency signature $S(p) = \{15,24,$ $16,6,25\}$, or ${PYQGZ}$.
\end{example}

Consider the set of words of length $n$ that avoid the pattern $p$. These words can be grouped into orbits since a word avoids $p$ if and only if all words in its orbit also avoid $p$. These orbits all have size $q$. Moreover, each orbit corresponds to a unique adjacency signature of length $n-1$, since two equivalent words (and thus all words in the orbit) have the same adjacency signature. This forms a $q$-to-1 bijection between words and adjacency signatures.

Now, as usual, let $\mathcal{A}(n,\{p\})$ denote the number of words of length $n$ which avoid the pattern $p$ (i.e., avoid all words in the orbit represented by $p$) and $A(n,\{S(p)\})$ denote the number of words which avoid the word $S(p)$. The following claim shows a bijective connection between the two sets.

\begin{theorem}
Let $n$ be a nonnegative integer, and $p$ a pattern of length $\ell$.
\begin{itemize}
    \item If $n=0$, we get $\mathcal{A}(n,\{p\}) = 1$ and $\mathcal{T}(n,\{p\})=0$.
    \item If $n=1$, we get $\mathcal{A}(n,\{p\}) = q$ when $\ell \geq 2$, and $0$ otherwise; $\mathcal{T}(n,\{p\})$ is $0$ when $\ell \geq 2$ and $q$ otherwise.
    \item If $n \geq 2$, we get $\mathcal{A}(n,\{p\}) = q A(n-1,\{S(p)\})$; similarly, we have $\mathcal{T}_p(n,\{p\}) = q T_{S(p)}(n-1, \{S(p)\})$. 
\end{itemize}

\end{theorem}

\begin{proof}
The first two statements are easily verifiable edge cases, so we focus on the last statement. We also focus on the avoiding function $\mathcal{A}$; the proof of the second part is similar. The key fact to note is that a word of length $n$ avoids a pattern $p$ if and only if its adjacency signature avoids $S(p)$.

Consider the aforementioned $q$-to-$1$ map from words to adjacency signatures, formed by grouping words into orbits of size $q$. The number of orbits whose words avoid $p$ is exactly the number of adjacency signatures that avoid $S(p)$. Since any word of length $n-1$ is a valid adjacency signature, the latter quantity is simply $A(n-1, \{S(p)\})$. 

Therefore, since each orbit has $q$ words, there are $qA(n-1, \{S(p)\})$ words that avoid $p$ as desired.
\end{proof}

Due to this theorem, any result that we have for words can be extended to patterns of a cyclic group. For instance, we have the following statement on generating functions.

\begin{corollary}
For any pattern $p$, we have the relations
\begin{itemize}
    \item $\mathcal{G}(z,\{p\}) = 1 + qzG(z,\{S(p)\})$; and
    \item $\mathcal{G}_p(z,\{p\}) = qzG_{S(p)}(z,\{S(p)\})$.
\end{itemize}
\end{corollary}

Actually, we can write the correlation for patterns explicitly in terms of the correlation of adjacency signatures.

\begin{proposition}
Let $p_1$ be a pattern with length $\ell$, and $p_2$ be another pattern. For $i \leq \ell - 2$, the entry $\mathcal{C}_i(p_1,p_2)$ is exactly $C_i(S(p_1), S(p_2))$. In addition, we always have $\mathcal{C}_{\ell-1}(p_1,p_2) = 1$.
\end{proposition}

\begin{proof}
Note first that each entry $\mathcal{C}_i(p_1,p_2)$ is either $0$ or $1$, since the stabilizer of every word is just $1$. In addition, the correlation entry is $0$ or $1$ depending on whether or not $p_1(1,\ell-i)$ is equivalent to $p_2(i+1,\ell)$. This is exactly whether or not $S(p_1(1,\ell-i)) = S(p_1)(1,\ell-i-1)$ is equal to $S(p_2(i+1,\ell)) = S(p_2)(i+1,\ell-1)$. The first statement of proposition quickly follows.

The second statement is true since any two letters are equivalent, so the last entry of any correlation is $1$.
\end{proof}

Using this comparison, we may express the Conway leading number for patterns in terms of the CLN of their adjacency signatures.

\begin{corollary}
The Conway leading number between a pattern $p$ of length $\ell$ and another pattern $p'$ is 
\[p \mathcal{L} p' = 1 + qS(p)\L S(p').\]
In particular, the expected wait time of $p$ is equal to $1 + qS(p) \L S(p)$.
\end{corollary}

The expected wait time result is not surprising, since in our random output, every letter starting from the second letter adjoins a letter to the output's adjacency signature. So in terms of the expected wait time, generating a random output under the group action is equivalent to generating a random word with one less letter in the original game.

Finally, all strategies for the original Penney's game carry over in their entirety. 

\begin{corollary}
If Alice picks a pattern with signature $S(p_1) = s_1(1)s_1(2) \dots s_1(\ell-1)$, then Bob's best strategy is to pick a pattern $p_2$ whose adjacency signature is of the form $S(p_2) = s^*s_1(1)s_1(2)\dots s_1(\ell-2)$; namely, Bob picks a $p'$ for which $S(p'(2,\ell)) = S(p(1,\ell-1))$. This is a winning strategy.
\end{corollary}

Specifically, the game is still non-transitive; Bob always has a winning strategy. In particular, if Alice picks the patterns ${aaab}$, ${abbb}$, ${abcc}$, ${aabc}$, then Bob picks the patterns ${abbb}$, ${abcc}$, ${aabc}$, ${aaab}$ to have a higher chance of winning, respectively. These patterns have adjacency signatures matching our first non-transitive example from Section~\ref{sec:originalproblem}.

\section{Symmetric group}\label{sec:symmetric}

We now look at the group action generated by the symmetric group $G = S_q$, where each element permutes letters.

For example, for $q=2$, the unique orbits are $\{AA,BB\}$ and $\{AB,BA\}$, represented by patterns $aa$ and $ab$. This case is covered in Section~\ref{sec:cyclic} as $S_2 \cong \mathbb{Z}_2$. 

For $q = 3$, the orbits are
\begin{itemize}
    \item $[AAA] = \{AAA,BBB,CCC\}$, represented by $aaa$;
    \item $[AAB] = \{AAB, AAC, BBA, BBC, CCA, CCB\}$, represented by $aab$;
    \item $[ABB] = \{ABB,ACC,BAA,BCC,CAA,CBB\}$, represented by $abb$;
    \item $[ABA] = \{ABA, ACA, BAB, BCB, CAC, CBC\}$, represented by $aba$; and
    \item $[ABC] = \{ABC, ACB, BAC, BCA, CAB, CBA\}$, represented by $abc$.
\end{itemize}

\subsection{Generating functions}

Recall in Section~\ref{sec:GroupActionAndCorrelation}, we show that the entry $\mathcal{C}_i$ of the autocorrelation $\mathcal{C}(p,p)$ of a pattern is either $0$ or $j!/k!$, where $j$ is the number of unused letters in $p$ and $k$ is the number of unused letters in $p(1,\ell-i)$. Equipped with this, we may compute a few example generating functions.

\begin{example}
Let $\mathcal{A} = \{a_1,a_2,\dots,a_q\}$. Consider the pattern $a_1 a_2 \dots a_q$, or the orbit consisting of words of length $q$ with all letters distinct. First, note that there are $q!$ elements in this orbit. In addition, we note
$$\mathcal{C}(a_1 a_2 \dots a_q, a_1 a_2 \dots a_q) = (0!, 1!, \dots, (q-1)!).$$
Thus, the generating functions for this pattern is
$$\mathcal{G}(z) = \frac{0! + 1!z + \dots + (q-1)!z^{q-1}}{q!z^{q} + (1-qz)(0! + 1!z + \dots + q!z^{q-1})},$$
$$\mathcal{G}_p(z) = \frac{q!z^{q}}{q!z^{q} + (1-qz)(0! + 1!z + \dots + q!z^{q-1})}.$$
\end{example}

\begin{example}
For $q=3$, we get $\mathcal{C}(abc,abc) = (1,1,2)$, from which it follows that
\begin{align*}
\mathcal{G}(z) &= \frac{1+z+2z^2}{6z^3 + (1-3z)(1+z+2z^2)} = \frac{1+z+2z^2}{1-2z-z^2} \\
&= 1 + 3 z + 9 z^2 + 21 z^3 + 51 z^4 + 123 z^5 + 297 z^6 + \cdots; \\
\mathcal{G}_p(z) &= \frac{6z^3}{6z^3 + (1-3z)(1+z+2z^2)} = \frac{6z^3}{1-2z-z^2} \\
&= 6 z^3 + 12 z^4 + 30 z^5 + 72 z^6 + 174 z^7 + 420 z^8 + \cdots.
\end{align*}
\end{example}

\subsection{Lower bound on the CLN}

We may also use these results to strengthen the lower bound of a pattern's CLN, derived in Section~\ref{sec:GroupActionAndCorrelation}, from within the interval $[q^{\ell-1}+1, q^{\ell-1} +q-1]$ to an exact bound.

\begin{proposition}
Fixing the group $G = S_q$ to form our group action, the least possible CLN for a pattern $p$ of length $\ell$ is exactly $q^{\ell-1} + q - 1$.
\end{proposition}

\begin{proof}
We showed in Section~\ref{sec:GroupActionAndCorrelation} that the minimum is between $q^{\ell-1} + 1$ and $q^{\ell-1} + q - 1$, so it suffices to show $p \mathcal{L} p \geq q^{\ell-1} + q - 1$ for all $p$. Since every letter of $\mathcal{A}$ is within a single orbit, the last entry $\mathcal{C}_{\ell-1}$ of the autocorrelation vector must be nonzero, so it is equal to $|G_{p(1)}|/|G_p| = (q-1)!/|G_p|$, which is either equal to $1$ or at least $q-1$.

Suppose $\mathcal{C}_{\ell-1} = 1$, meaning that the stabilizer of the last letter of $p$ has the same order as the stabilizer of $p$. This can only mean that $p = aa\dots a$ for some letter $a \in \mathcal{A}$, which has CLN $q^{\ell-1} + \dots + q + 1 \geq q^{\ell-1} + q - 1$ for $q \geq 2$. On the other hand, if $\mathcal{C}_{\ell-1} \neq 1$, then $\mathcal{C}_{\ell-1}$ is at least $q-1$ and the CLN is at least $q^{\ell-1} + q - 1$. 

To show achievability, note the pattern $p = aa\dots ab$ works. The bound is therefore sharp, and the claim follows.
\end{proof}

We also partially characterize which patterns achieve this lower bound with the following corollary.

\begin{corollary}
For $q \geq 4$, the lower bound is only achieved by words with two distinct letters.
\end{corollary}

\begin{proof}
If $p$ only contains a single distinct letter, then it obviously it doesn't . Now suppose that $p$ contains at least three distinct letters. Then $|G_p| \leq (q-3)!$ and $\mathcal{C}_{\ell-1} \geq (q-1)(q-2) > (q-1)$, implying $p\mathcal{L} p > q^{\ell-1} + q - 1$. Therefore, to achieve the lower bound, we must use exactly two distinct letters.
\end{proof}

\begin{remark*}
Like with non-overlapping words, it seems hard to exactly characterize patterns with CLN $q^{\ell-1} + q - 1$. In particular, for $(q,\ell) = (4,5)$, the lower bound is achieved by the patterns $aaaab$, $aaaba$, $aabab$, $abaaa$, $ababb$, and $abbbb$.
\end{remark*}

\begin{remark*}
The corollary also holds for $q=2$. But the case $q=3$ is special; there are patterns with three distinct letters that achieve the minimal CLN. For instance, with $(q,\ell) = (3,4)$, the following patterns have a CLN of $3^{4-1} + 3-1 = 29$: $aaab, aaba, aabc, abaa, abbb, abcc$.
\end{remark*}

\subsection{Odds}

Just like the original Penney's game, even if the wait time of $p_2$ exceeds the wait time of $p_1$, Bob may still win by picking $p_2$ if Alice picks $p_1$.

\begin{example}
For $(q, \ell) = (4,4)$, the expected wait time for the patterns $p_1 = aabc$ and $p_2 = abbc$ are $\frac4{24} \cdot p_1\mathcal{L} p_1 = \frac{35}{3}$ and $\frac4{24} \cdot p_2\mathcal{L} p_2 = 26$, respectively. However, in a randomly generated string of letters, $p_2$ appears before $p_1$ with odds
\[\frac{24}{24} \cdot \frac{p_1\mathcal{L} p_1 - p_1\mathcal{L} p_2}{p_2\mathcal{L} p_2 - p_2\mathcal{L} p_1} = \frac{7}{5}.\]
\end{example}
	
In Figure~\ref{fig:44-graph} we let $(q, \ell) = (4,4)$ and show every pattern $p$ with its best beater $p'$, denoted as $p \to p'$. We label each arrow with the odds of the second pattern winning. The data for this graph was generated with a program, found here: \url{https://github.com/seanjli/penneys-game-patterns}.

\begin{figure}[ht!]
    \centering
    \includegraphics[width=10cm]{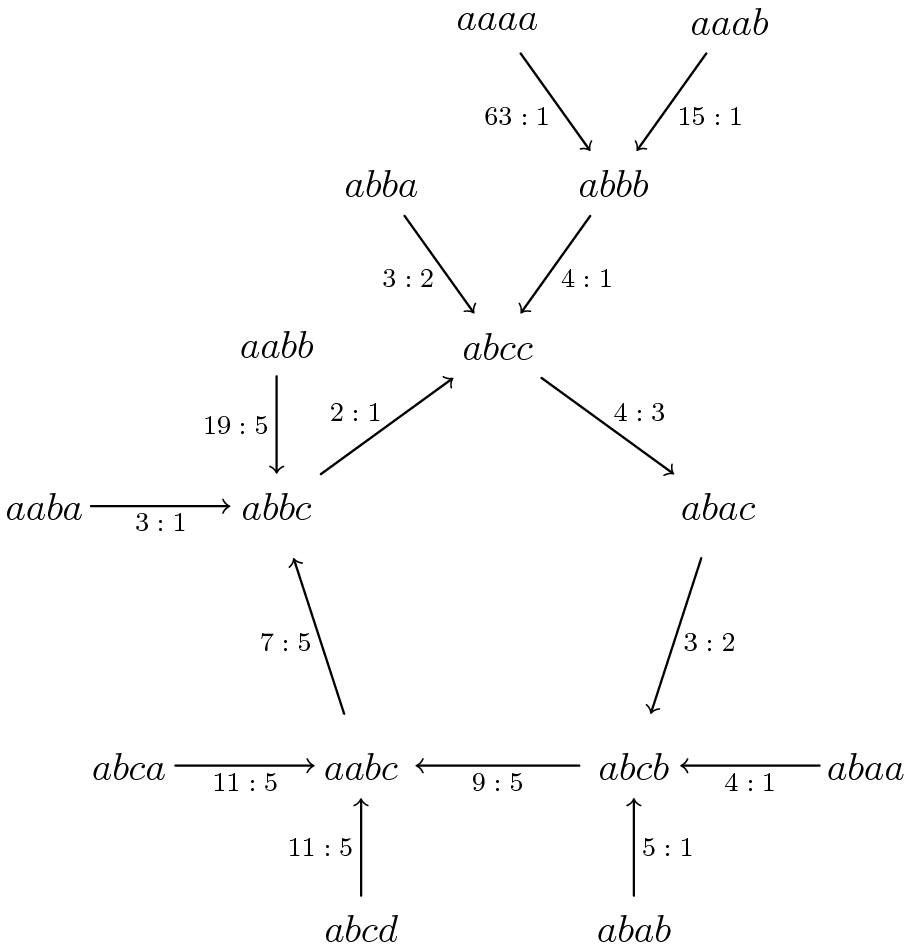}
    \caption{Directed graph of Bob's best choices for $(q, \ell) = (4,4)$.}
    \label{fig:44-graph}
\end{figure}

Note that for this choice of $q$ and $\ell$, the game is  non-transitive. Namely, in Figure~\ref{fig:44-graph}, we have a non-transitive cycle of length $5$:
\[aabc \ \stackrel{7:5}{\to}\ abbc\ \stackrel{2:1}{\to}\ abcc\ \stackrel{4:3}{\to}\ abac\ \stackrel{3:2}{\to}\ abcb\ \stackrel{9:5}{\to}\ aabc.\]

Unlike the original Penney's game, in the game with patterns, Alice can sometimes win. Suppose Alice picks the pattern $p_1$ and Bob picks $p_2$.

\begin{proposition}
For $\ell < (q+2)/2$, the pattern $p_1 = a_1 a_2 \dots a_\ell$ (i.e. a pattern with $\ell$ different letters) has better odds against any other pattern of the same length.
\end{proposition}

\begin{proof}
Recall that the odds that Bob wins are exactly
\[\frac{\frac{|G|}{r_1}p_1\mathcal{L}p_1 - \frac{|G|}{r_1}p_1\mathcal{L}p_2}{\frac{|G|}{r_2}p_2\mathcal{L}p_2 - \frac{|G|}{r_2}p_2\mathcal{L}p_1}.\]

We replace $\frac{|G|}{r_a}\mathcal{C}_i(p_a,p_b) = |G|\mathcal{C}^*_i(p_a,p_b)$ for sake of brevity. We scale the numerator and denominator by $|G|$: for example,
\[\frac{|G|}{r_1}p_1\mathcal{L}p_1 = \sum_{i=0}^{\ell-1} \frac{|G|}{r_1}\mathcal{C}_i(p_1,p_1) q^{\ell-1-i} = \sum_{i=0}^{\ell-1}|G|\mathcal{C}^*_i(p_1,p_1) q^{\ell-1-i}\]
and we know each coefficient $|G|\mathcal{C}^*_i(p_1,p_1)$ is either $0$ or $|G_{p_1(1,\ell-i)}|$. A similar simplification happens for each of the other CLNs.

We claim if Alice chooses $p_1 = a_1 a_2 \dots a_\ell$, the odds that Bob wins are always less than $1$.

First note $|G|\mathcal{C}^*_i(p_1,p_1) = (q-\ell+i)!$, yielding
\[\frac{|G|}{r_1} p_1 \mathcal{L} p_1 = (q-\ell)! q^{\ell-1} + (q-\ell+1)! q^{\ell-2} + \dots + (q-1)!.\]
Since $\frac{|G|}{r_1} p_1 \mathcal{L} p_2 \geq 0$, the numerator is at most the right-hand quantity above.

We now focus on the denominator. To begin, note that Bob's pattern has at most $\ell - 1$ distinct letters: if all letters are different, then his pattern is equivalent to Alice's pattern which is prohibited. Thus $|G|\mathcal{C}^*_0(p_2,p_2)= |G_{p_2}| \geq (q-\ell+1)!$. In addition, for any pattern $p_2$ we have $|G|\mathcal{C}^*_{\ell-1}(p_2,p_2) = (q-1)!$, so we obtain the lower bound:
\[\frac{|G|}{r_2} p_2 \mathcal{L} p_2 \geq |G|\mathcal{C}^*_0(p_2,p_2) q^{\ell-1} + |G|\mathcal{C}^*_{\ell-1}(p_2,p_2) \geq (q-\ell+1)! q^{\ell-1} + (q-1)!\]
Finally, note that $|G|\mathcal{C}^*_i(p_2,p_1)$ is equal to $0$, if $i = 0$; and either $0$ or $|G_{p_1(i+1,\ell)}| = (q-\ell+i)!$.
This gives us an upper bound of $(q-\ell+i)!$ on each coefficient $|G|\mathcal{C}^*_i(p_2,p_1)$, so
\[\frac{|G|}{r_2} p_2 \mathcal{L} p_1 \leq \sum_{i=1}^{\ell-1} (q-\ell+i)! q^{\ell-1-i} = (q-\ell+1)! q^{\ell-2} + \dots + (q-1)!.\]
Thus, the denominator $\frac{|G|}{r_2} p_2\mathcal{L} p_2 - \frac{|G|}{r_2} p_2\mathcal{L}p_1$ is at least
\[(q-\ell+1)!q^{\ell-1} - \sum_{i=1}^{\ell-2} (q-\ell+i)! q^{\ell-1-i}.\]

Combining the bounds for the numerator and denominator, we may bound the winning odds for Bob from above:
\begin{align*}
    \frac{\frac{|G|}{r_1}p_1\mathcal{L}p_1 - \frac{|G|}{r_1}p_1\mathcal{L}p_2}{\frac{|G|}{r_2}p_2\mathcal{L}p_2 - \frac{|G|}{r_2}p_2\mathcal{L}p_1} &\leq \frac{(q-\ell)! q^{\ell-1} +  \sum_{i=1}^{\ell-2} (q-\ell+i)! q^{\ell-1-i}}{(q-\ell+1)!q^{\ell-1} - \sum_{i=1}^{\ell-2} (q-\ell+i)! q^{\ell-1-i}} \\
    &= \frac{(q-\ell)! q^{\ell-1} +  (q-\ell+1)!q^{\ell-1}}{(q-\ell+1)!q^{\ell-1} - \sum_{i=1}^{\ell-2} (q-\ell+i)! q^{\ell-1-i}} - 1.
\end{align*}
Using the inequality $(q-\ell+i)! q^{\ell-1-i} < (q-\ell+1)! q^{\ell-2}$, we may bound the right-hand side from above by decreasing the denominator. Namely, the RHS is at most
\[\frac{(q-\ell)! q^{\ell-1} +  (q-\ell+1)!q^{\ell-1}}{(q-\ell+1)!q^{\ell-1} - (\ell-2)(q-\ell+1)! q^{\ell-2}} - 1 = \frac{\ell-1}{q-\ell+1}.\]
The right-hand side is less than $1$ for $\ell < (q+2)/2$. Thus, Bob's odds of winning are always less than $1$, and he has a disadvantage.
\end{proof}

\begin{example}
Fix $(q,\ell) = (6,3)$, so $\ell < (q+2)/2$. Then if Alice picks $abc$, Bob has unfavorable odds no matter what pattern he chooses. Table~\ref{tab:alice-wins-63} shows the odds of Bob winning for all possible choices for Bob when Alice picks $abc$.

\begin{table}[ht!]
    \centering
    \begin{tabular}{c|c}
    \textbf{Pattern} & \textbf{Odds} \\ \hline\hline
    $aaa$ & $1 : 14$ \\
    $aab$ & $1 : 2$ \\
    $aba$ & $1 : 4$ \\
    $abb$ & $1 : 4$
    \end{tabular}
    \caption{For $(q,\ell) = (6,3)$, Bob has losing odds against Alice no matter what pattern he picks.}
    \label{tab:alice-wins-63}
\end{table}

\end{example}

\section{Acknowledgements}

We are grateful to the MIT PRIMES-USA program for giving us the opportunity to conduct this research.

\end{document}